\documentclass[11pt]{amsart}

\usepackage{latexsym}
\usepackage{amsmath}
\usepackage{amssymb}
\usepackage{graphics}

\usepackage{epsfig}
\usepackage{epic}

\usepackage{enumerate}

\usepackage{color}

\newtheorem{theorem}{Theorem}[section]
\newtheorem{lemma}[theorem]{Lemma}
\newtheorem{corollary}[theorem]{Corollary}
\newtheorem{proposition}[theorem]{Proposition}

\newcommand{\cN}{{\mathcal N}}
\newcommand{\cS}{{\mathcal S}}

\newcommand{\id}{{\rm id}}

\newcommand{\blue}{\textcolor{black}}
\newcommand{\Blue}{\textcolor{black}}
\newcommand{\green}{\textcolor{black}}
\newcommand{\red}{\textcolor{black}}

\begin{document}

\title[Defining networks using ancestral profiles]{Defining phylogenetic networks using ancestral profiles}

\author{Allan Bai}
\address{School of Mathematics and Statistics, University of Canterbury, Christchurch, New Zealand}
\email{amb337@uclive.ac.nz}

\author{P\'{e}ter L.\ Erd\H{o}s} 
\address{Alfr\'{e}d R\'{e}nyi Institute of Mathematics, Hungarian Academy of Sciences, Budapest, Hungary}
\email{erdos.peter@renyi.hu}

\author{Charles Semple}
\address{School of Mathematics and Statistics, University of Canterbury, Christchurch, New Zealand}
\email{charles.semple@canterbury.ac.nz}

\author{Mike Steel}
\address{School of Mathematics and Statistics, University of Canterbury, Christchurch, New Zealand}
\email{mike.steel@canterbury.ac.nz}

\thanks{The first, third, and fourth authors were supported by the New Zealand Marsden Fund (UOC1709). The second author was supported in part by the National Research, Development and Innovation Office (NKFIH grants K~116769 and KH~126853).}

\keywords{Tree-child networks, orchard networks, accumulation phylogenies, ancestral profiles, path-tuples}

\subjclass{05C85, 92D15}

\date{\today}

\begin{abstract}
Rooted phylogenetic networks provide a more complete representation of the ancestral relationship between species than phylogenetic trees when reticulate evolutionary \blue{processes} are at play. One way to reconstruct a \blue{phylogenetic} network is to consider  its `ancestral profile' (the number of paths from each ancestral vertex to each \blue{leaf}). In general, this information does not uniquely determine  the underlying phylogenetic network. A recent paper considered a new class of phylogenetic networks \blue{called} `orchard networks' where this uniqueness was claimed to hold. Here we show that an additional restriction on the network, \blue{that of being} \Blue{`stack-free'}, is required in order for the original uniqueness claim to hold. On the other hand, if the additional stack-free restriction is lifted, we establish an alternative result; namely, there is uniqueness within the class of orchard networks up to the resolution of vertices of high in-degree.
\end{abstract}

\maketitle

\section{Introduction}

Evolutionary relationships between species are generally represented by phylogenetic trees, where the species at the present appearing as the leaves of the tree, and ancestral species corresponding to interior vertices.  Over the last several decades, a wide variety of methods have been developed for reconstructing phylogenetic trees from genomic data \cite{fel04}, and these are now widely used in large-scale studies in systematic biology (e.g.,~\cite{jet12, uph19}) and associated fields (e.g, in  epidemiology to classify strains of viruses such as HIV, influenza, and SARS-Cov2~\cite{wor20}). However, for certain groups of organisms, the tree model is overly simplistic. This is because of the intricacies of ancestral processes whereby lineages not only split, but sometimes combine together to form new lineages. This latter pattern of evolution is collectively referred to as `reticulation', and includes the formation of hybrid species, horizontal gene transfer, and endosymbiosis events~\cite{doo99, hus10, koo15}. Consequently, certain portions of the `Tree of Life' are better described by a phylogenetic network that explicitly exhibits reticulation events. Although there is a well-developed theory for reconstructing phylogenetic trees from various types of data~\cite{fel04, sem03}, phylogenetic network reconstruction is much more subtle. In particular, for certain types of data it is impossible to distinguish between different (non-isomorphic) phylogenetic networks~\cite{par15}. One way to address this non-identifiability issue is to work within a subclass of phylogenetic networks that includes phylogenetic trees along with phylogenetic networks that are sufficiently tame. An example is the class of `normal' networks, for which certain \blue{reconstructive} results have been established~\cite{bor18, lin20, wil08, wil10}.  The slightly more general class of `tree-child' networks also allows for unique reconstruction from various types of data~\cite{car08, car09}.
 
\blue{In this paper}, we focus on the unique \blue{reconstruction} of networks  from  their  `ancestral profile', which, roughly speaking, is the number of paths from each ancestral vertex in the network to each extant leaf. It was shown that all \blue{binary} `tree-sibling time-consistent' and all \blue{binary} `tree-child' networks are uniquely determined by their ancestral profile~\cite{car08, car09}. In a recent paper~\cite{erd19}, this result was extended to \blue{the} larger class of \blue{binary} `orchard networks', which allows for an unbounded number of vertices in the network for a given \blue{number of leaves. This contrasts with the classes of binary `tree-sibling time-consistent' and binary} `tree-child' networks, for which the size of the network is bounded by the number of leaves. However, the result in~\cite{erd19} omitted an extra condition required for unique reconstruction, namely, the network cannot contain a tower \blue{(`stack')} of reticulations. We show here that this `stack-free' condition is necessary, and that when this extra condition is included the original result claimed in~\cite{erd19} holds. Moreover, this \blue{result} then generalises \blue{(Theorem~\ref{main1})} to the class of \blue{stack-free orchard} networks \blue{in which} reticulate vertices \blue{are allowed} to have arbitrarily high in-degree. \blue{Note that, the uniqueness is amongst all phylogenetic networks with vertices of arbitrarily high in-degree, that is,} the ancestral profile of a stack-free orchard network is always different to the ancestral profile of any other phylogenetic network, even if it is \blue{neither} orchard \blue{nor} stack-free. When the stack-free condition is lifted, we describe a second result (Theorem~\ref{main2}) which states that, within the class of orchard networks, \blue{the} ancestral profile \blue{of an orchard network} uniquely determines the orchard \blue{network} up to the resolution of vertices of high in-degree.

The structure of the paper is as follows. \blue{The next section recalls} definitions of phylogenetic networks (which are permitted here to contain vertices of high in-degree), along with the notion of ancestral profile. We \blue{also} describe the class of orchard networks. \blue{This class} was introduced and studied independently \blue{in}~\cite{erd19} (for binary \blue{networks}) and~\cite{jan20} (for networks that allow high in-degree). In Section~\ref{results}, we turn to the question of whether the ancestral profile of an orchard network determines that network (either within the class of orchard networks, or more generally), \blue{and state the two main results of the paper. The first main result, Theorem~\ref{main1}, states} a corrected form of~\cite[Theorem~2.2]{erd19} for stack-free networks. The necessary adjustments required for \blue{the proof of Theorem~\ref{main1} are given} in the Appendix. \blue{The second main result, Theorem~\ref{main2},} is a \blue{reconstructive} result that holds when the stack-free condition is removed. \blue{Additionally, we discuss the relationships between Theorem~\ref{main1} and the main results in~\cite{car08, car09}.} \blue{The proof of Theorem~\ref{main2} is given in Section~\ref{proof}. Some concluding comments are given in Section~\ref{comments}, the last section of the paper.}

\section{Preliminaries}

Throughout the paper $X$ denotes a non-empty finite set and, unless otherwise stated, all paths are directed.
For sets $A$ and $B$, we denote the set obtained from $A$ by removing every element in $A$ that is also in $B$ by $A-B$. \blue{Furthermore, if $(u, v)$ is an arc of an acyclic directed graph, we say $u$ is a {\em parent} of $v$.}

\noindent {\bf Phylogenetic networks.} The following definition of phylogenetic network is slightly more general than in~\cite{erd19}. A {\em phylogenetic network on $X$} is a rooted acyclic directed graph with no arcs in parallel and satisfying the following properties:
\begin{enumerate}[(i)]
\item the (unique) root has in-degree zero and out-degree two;

\item a vertex with out-degree zero has in-degree one, and the set of vertices with out-degree zero is $X$; and

\item all other vertices either have in-degree one and out-degree two, or in-degree at least two and out-degree one.
\end{enumerate}
\noindent We will refer to a phylogenetic network in which every vertex has in-degree at most two as a {\em binary} phylogenetic network. 

We pause to make two technical remarks. First,  if $|X|=1$, we additionally allow a single vertex to be a phylogenetic network, in which case, the root is the vertex in $X$. 
Second, suppose that $\cN_1$ and $\cN_2$ are two phylogenetic networks on $X$ with vertex and arc sets $V_1$ and $E_1$, and $V_2$ and $E_2$, respectively. We say $\cN_1$ is {\em isomorphic} to $\cN_2$ if there exists a bijection $\varphi: V_1\rightarrow V_2$ such that $\varphi(x)=x$ for all $x\in X$, and $(u, v)\in E_1$ if and only if $(\varphi(u), \varphi(v))\in E_2$ for all $u, v\in V_1$.

Let $\cN$ be a phylogenetic network on $X$. The vertices with out-degree zero are the {\em leaves} of $\cN$, and so $X$ is called the {\em leaf set} of $\cN$. Furthermore, vertices with in-degree one and out-degree two are {\em tree vertices}, while vertices of in-degree \blue{at least} two and out-degree one are {\em reticulations}. The arcs directed into a reticulation are called {\em reticulation arcs}, all other arcs are {\em tree arcs}. \blue{To illustrate, a phylogenetic network on $\{x_1, x_2, \ldots, x_6\}$ is shown in Fig.~\ref{orchard}. Vertices $u$ and $v$ are reticulations, while vertex $w$ is a tree vertex. Here, as throughout the paper, all arcs are directed down the page.}

\begin{figure}
\center
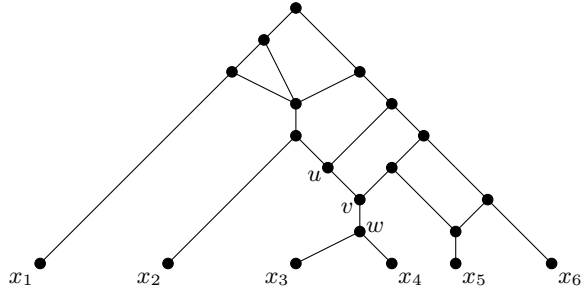
\caption{\blue{A phylogenetic network on $\{x_1, x_2, \ldots, x_6\}$. The $2$-element set $\{x_3, x_4\}$ is a cherry, while the ordered pair $\{x_5, x_6\}$ is a reticulated cherry, in which $x_5$ is the reticulation leaf.}}
\label{orchard}
\end{figure}

\noindent {\bf Ancestral tuples and ancestral profile.} Let $\cN$ be a phylogenetic network on $X$ with vertex set $V$. Let $v_1, v_2, \ldots, v_t$ be a fixed (arbitrary) labelling of the vertices in $V-X$. For all $x\in X$, the {\em ancestral tuple} of $x$, denoted $\sigma(x)$, is the $t$-tuple whose $i$-th entry is the number of paths in $\cN$ from $v_i$ to $x$. Denoted by $\Sigma_{\cN}$, we call \red{the set}
$$\Sigma_{\cN}=\{\red{(x, \sigma(x))}: x\in X\},$$
\red{of ordered pairs} the {\em ancestral profile} of $\cN$. Furthermore, if $\cN'$ is a phylogenetic network on $X$ and, up to an ordering of the non-leaf vertices of $\cN'$, we have $\Sigma_{\cN'}=\Sigma_{\cN}$, we say {\em $\cN'$ realises $\Sigma_{\cN}$}. Lastly, although $\Sigma_{\cN}$ depends on the ordering of the vertices in $V-X$, the ordering is fixed and so the labelling can be effectively ignored.

To illustrate these notions consider the two networks $\cN$ and $\cN'$ shown in Fig.~\ref{fig1}.  Under the labelling of the non-leaf vertices of $\cN$ shown, we have
\begin{align*}
\Sigma_{\cN}=\{(x_1, & (1, 1, 0, 0, 0, 0, 0)), (x_2, (1, 0, 1, \blue{1}, 0, 0, 0)), \\
 & (x_3, (1, 0, 1, 0, 1, 0, 0)), (x_4, (3, 1, 2, 1, 1, 1, 1)) \}.
 \end{align*}
The other network $\cN'$ in Fig.~\ref{fig1} also realises $\Sigma_{\cN}$, because under the  ordering of the non-leaf vertices of $\cN'$ shown in this figure, we have
$\Sigma_{\cN'}=\Sigma_{\cN}$. On the other hand, $\cN$ and $\cN'$ are not isomorphic. To see this observe that the parent of $x_2$ in $\cN$ has a unique path  to $x_4$ of length~$3$, while the parent of $x_2$ in $\cN'$ also has a unique path to $x_4$ but this path has length~$2$.

\begin{figure}
\center
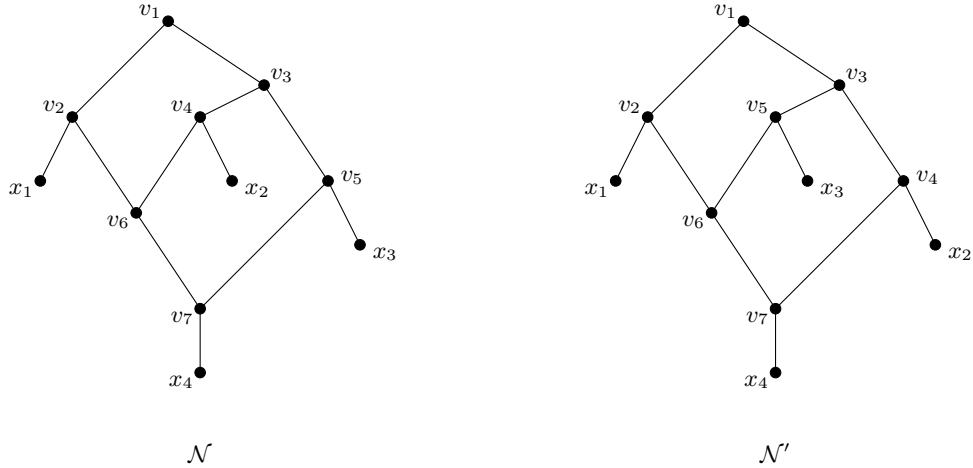
\caption{Two binary networks $\cN$ and $\cN'$ with the same ancestral profile (for the labelling of the vertices in $V-X$ shown). However, $\cN$ and $\cN'$ are not isomorphic. }
\label{fig1}
\end{figure}

\noindent {\bf Cherries and reticulated cherries.} Let $\mathcal N$ be a phylogenetic network on $X$, and let $\{a, b\}$ be a $2$-element subset of $X$. Let $p_a$ and $p_b$ denote the parents of $a$ and $b$, respectively. We say $\{a, b\}$ is a {\em cherry} of $\mathcal N$ if $p_a=p_b$. Furthermore, if one of the parents, say $p_b$, is a reticulation and $(p_a, p_b)$ is an arc in $\mathcal N$, then $\{a, b\}$ is a {\em reticulated cherry} of $\cN$, in which case, $b$ is the {\em reticulation leaf} of the reticulated cherry. Observe that $p_a$ is necessarily a tree vertex. \blue{As an example, in Fig.~\ref{orchard}, $\{x_3, x_4\}$ is a cherry, while $\{x_5, x_6\}$ is a reticulated cherry with $x_5$ as the reticulation leaf.}

We next describe two operations associated with cherries and reticulated cherries that are central to this paper. Let $\cN$ be a phylogenetic network. First suppose that $\{a, b\}$ is a cherry of $\cN$. Then {\em reducing} $b$ is the operation of deleting $b$ and suppressing the resulting vertex of in-degree one and out-degree one. If the parent of $a$ and of $b$ is the root of $\cN$, then reducing $b$ is the operation of deleting $b$ as well as deleting the root of $\cN$, thus leaving only the isolated vertex $a$. Now suppose that $\{a, b\}$ is a reticulated cherry of $\cN$ in which $b$ is the reticulation leaf. Then {\em cutting} $\{a, b\}$ is the operation of deleting the reticulation arc joining the parents of $a$ and $b$, and suppressing \blue{any} resulting vertices of in-degree one and out-degree one. \blue{Note that the parent of $a$ is always suppressed. However, the parent of $b$ is suppressed only if its in-degree in $\cN$ is exactly two.} It is easily seen that the operations of reducing a cherry and cutting a reticulated cherry both result in a phylogenetic network. Collectively, we refer to these two operations as {\em cherry reductions}. 

\noindent {\bf Orchard networks.} For a phylogenetic network $\cN$, the sequence
\begin{align}
\cN=\cN_0, \cN_1, \cN_2, \ldots, \cN_k
\label{seq1}
\end{align}
of phylogenetic networks is a {\em cherry-reduction sequence of $\cN$} if, for all $i\in \{1, 2, \ldots, k\}$, the phylogenetic network $\cN_i$ is obtained from $\cN_{i-1}$ by a (single) cherry reduction. The sequence is {\em maximal} if $\cN_k$ has no cherries or reticulated cherries. If $\cN_k$ consists of a single vertex, the sequence is {\em complete}. \Blue{If $\cN$ has a complete cherry-reduction sequence, then $\cN$ is} an {\em orchard network}. \blue{It is easily checked that the phylogenetic network shown in Fig.~\ref{orchard} is orchard.}

A fundamental property of orchard networks is that if one cherry-reduction sequence leads to a single vertex  (in which case $\cN$ is an orchard network), then every maximal cherry-reduction sequence leads to a single vertex (regardless of any choices made during the construction of a cherry-reduction sequence). This result was established for binary orchard networks in~\cite[Proposition 4.1]{erd19}, and independently shown to hold for general phylogenetic networks in~\cite[Theorem 1]{jan20}.

\begin{proposition}
\label{pro:orchard}
Let $\cN$ be an orchard network, and let
$$\cN = \cN_0, \cN_1, \ldots, \cN_\ell$$
be a maximal cherry-reduction sequence. Then this sequence is complete.
\end{proposition}

\section{Main Results}
\label{results}

\blue{In this section we state the two main results of the paper.} A {\em stack} in a phylogenetic network $\cN$ is a pair of reticulations, $u$ and $v$, such that one of the reticulations, say $u$, is a parent of the other; that is, $(u, v)$ is an arc of $\cN$. We refer to $(u, v)$ as a {\em stack arc} of $\cN$.   A phylogenetic network is said to be {\em stack-free} if it has no stacks.

It was claimed in~\cite[Theorem~2.2]{erd19} that, up to isomorphism, every binary orchard network is uniquely determined by its ancestral profile. However, Fig.~\ref{fig1} shows a pair of non-isomorphic phylogenetic networks that are both \blue{binary} orchard networks, and which have identical ancestral profiles. Notice that both networks in  Fig.~\ref{fig1} contain a stack, \blue{in particular, reticulations $v_6$ and $v_7$}. A corrected version of~\cite[Theorem~2.2]{erd19} \blue{is Theorem~\ref{main1}, the first main result of the paper}, which is now extended to allow phylogenetic networks \blue{with reticulations of in-degree at least two}. The proof follows the same argument as in~\cite{erd19}, but some adjustments are required to certain lemmas to allow for the generality beyond binary phylogenetic networks, and (at one point) to impose the stack-free requirement.  We describe the required adjustments to the original proof in the Appendix.

\begin{theorem}
Let $\cN$ be a stack-free orchard network on $X$ with vertex set $V$. Then, up to isomorphism, $\cN$ is the unique phylogenetic network on $X$ realising $\Sigma_{\cN}$.
\label{main1}
\end{theorem}

\blue{We now consider what can be said if the stack-free condition is lifted.} Let $\cN$ be a phylogenetic network on $X$ with vertex set $V$. Define a relation $\sim'$ on $V-X$ by writing $u\sim' v$ if $u$ and $v$ are reticulations and either $(u, v)$ or $(v, u)$ is an arc of $\cN$. Let $\sim$ be the transitive closure of $\sim'$; the equivalence classes of vertices under $\sim$ are called {\em sinks}. Thus, a phylogenetic network $\cN$ is stack-free if and only if each of its sinks has size 1 (i.e.\ each reticulation forms its own equivalence class). The {\em stack \blue{identification}} of $\cN$, \Blue{denoted} $\id(\cN)$, is the phylogenetic network obtained from $\cN$ by identifying all the vertices within each sink $S$ to a single vertex $v_S$ (and removing any arcs between vertices of the same sink). \blue{Observe that $\id(\cN)$ can be obtained from $\cN$ by repeatedly deleting each stack arc and identifying its end vertices.} Note that $\id(\cN)$ is not necessarily a phylogenetic network because it may have arcs in parallel. However, if $\cN$ is orchard, then, as we show in the next section (Lemma~\ref{lemo1}), $\id(\cN)$ is also orchard. \blue{To illustrate the notion of stack \blue{identification}, consider the two orchard networks $\cN$ and $\cN'$ shown in Fig.~\ref{fig1}. The stack \blue{identifications} of $\cN$ and $\cN'$ are shown in Fig.~\ref{stack}. Observe that $\id(\cN)\cong \id(\cN')$.}

\begin{figure}
\center
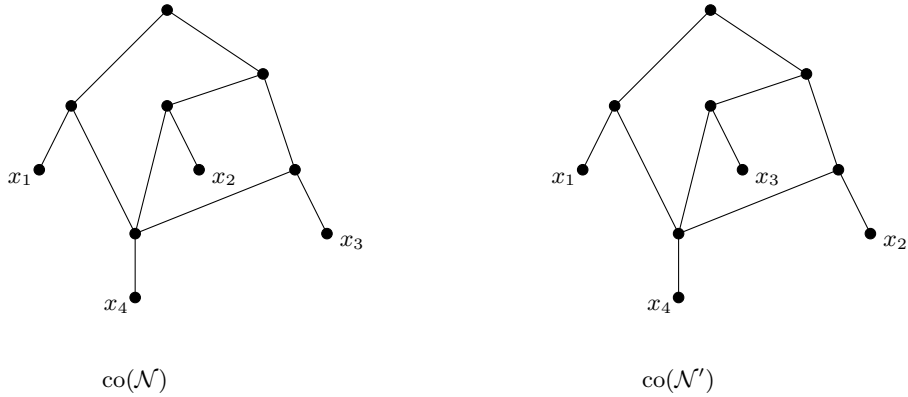
\caption{The stack \blue{identifications} $\id(\cN)$ and $\id(\cN')$ of the orchard networks $\cN$ and $\cN'$, respectively, shown in Fig.~\ref{fig1}. Observe that both $\id(\cN)$ and $\id(\cN')$ are orchard networks, and  $\id(\cN)\cong \id(\cN)$.}
\label{stack}
\end{figure}

\blue{The next theorem is the second main result of the paper.}

\begin{theorem}
Let $\cN$ and $\cN'$ be orchard networks on $X$. If $\cN$ and $\cN'$ have the same ancestral profile, then $\id(\cN)\cong \id(\cN')$.
\label{main2}
\end{theorem}

\blue{We end this section with a brief discussion concerning Theorem~\ref{main1} and its relationship with the main results in~\cite{car08, car09}.} Let $\cN$ be a phylogenetic network. We say $\cN$ is {\em tree-child} if every non-leaf vertex of $\cN$ is a parent of a tree vertex or a leaf. Furthermore, $\cN$ is {\em tree-sibling} if every reticulation has a parent that is also a parent of a tree vertex or a leaf. Also, $\cN$ is {\em time-consistent} if there is a map $t$ from the vertex set of $\cN$ to the non-negative integers having the property that if $(u, v)$ is an arc of $\cN$, then $t(u)=t(v)$ if $(u, v)$ is a reticulation arc; otherwise, $t(u)< t(v)$.

The class of stack-free orchard networks includes the class of tree-child networks as a proper subclass. \blue{(A proof that a binary tree-child network is orchard is given in~\cite{bor16}. The generalisation to allowing reticulations with in-degree more than two is straightforward.)} Moreover, although a tree-child network on \Blue{a} leaf set of size $n$ can have at most $n-1$ \Blue{reticulations}~\cite{car09}, a stack-free orchard network can have arbitrarily many \Blue{reticulations}, as indicated in Fig.~\ref{unbounded}. The class of stack-free orchard networks also includes tree-sibling time-consistent networks with no stacks. \blue{(A proof for the binary case is given in~\cite{erd19}. This proof generalises to allowing reticulations with in-degree more than two.)} Like tree-child networks, the number of reticulations of such a network \blue{on a leaf set of size $n$} is linear, in this case \blue{at most} $2n-4$~\cite{car08}.

\begin{figure}
\center
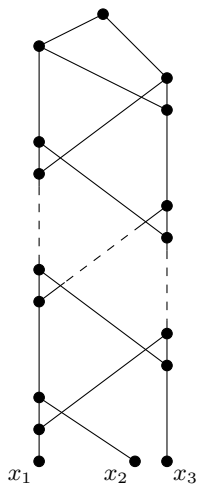
\caption{An orchard network that is binary and stack-free, for which the total number of vertices is not bounded by the size of its leaf set.}
\label{unbounded}
\end{figure}

Theorem~\ref{main1} in part generalises results in~\cite{car08, car09}\footnote{The results in~\cite{car08, car09} are slightly stronger than that described here as they allow tree vertices to have out-degree at least two.}. These papers consider the classes of time-sibling time-consistent and tree-child networks, respectively, in the context of a formation on $X$ for reconstruction \blue{that is equivalent to ancestral profile}. They establish uniqueness results for tree-child~\cite[Theorem~1]{car09} and for \blue{binary} tree-sibling time-consistent networks with no stacks~\cite[Theorem~6]{car08}. However, the uniqueness is within the respective classes. \blue{Thus, for example, in our terminology, it is shown in~\cite{car09} that if $\cN$ is a tree-child network on $X$, then, up to isomorphism, $\cN$ is the unique tree-child network on $X$ realising $\Sigma_{\cN}$.} 

\section{\blue{Proof of Theorem~\ref{main2}}}
\label{proof}

The proof of Theorem~\ref{main2} makes use of a sequence of lemmas. We begin by showing that the stack \blue{identification} of an orchard network is orchard.

\begin{lemma}
\blue{Let $\cN$ be an orchard network, and let $e$ be a stack arc of $\cN$. Suppose that $\cN'$ is obtained from $\cN$ by deleting $e$ and identifying its end vertices. Then $\cN'$ is an orchard network.}
\label{lemo1}
\end{lemma}

\begin{proof}
\blue{Let $e=(u, v)$. We first show that $\cN'$ has no parallel arcs, that is, $\cN'$ is a phylogenetic network. Assume $\cN'$ has two parallel arcs. Then, by the construction of $\cN'$, these arcs are directed out of a tree vertex $t$ and directed into the vertex, say $v'$, identifying $u$ and $v$, in which case, $(t, u)$ and $(t, v)$ are arcs of $\cN$.}
Since $\cN$ is orchard, $\cN$ has a complete cherry-reduction sequence $\cS$. \blue{Applying $\cS$ to $\cN$, this sequence eventually suppresses $u$ and $v$ via cutting a reticulated cherry. Clearly, $v$ is suppressed before $u$. Since $u$ is a reticulation parent of $v$, it follows that prior to $v$ being suppressed, $(t, v)$ is cut as part of a cherry reduction of $\cS$. But this requires $t$ to have a descendant leaf that is not a descendant of $v$. Since $u$ is the other child of $t$, there are no such leaves, and so $\cS$ is not a cherry-reduction sequence of $\cN$, a contradiction.}
Thus $\cN'$ has no parallel arcs.

We complete the proof by showing that $\cN'$ is orchard. Let
$$\cN=\cN_0, \cN_1, \cN_2, \ldots, \cN_k$$
be a complete cherry-reduction sequence for $\cN$. Since $u$ and $v$ are reticulations, and $u$ is a parent of $v$, it follows that, for some $i\in \{1, 2, \ldots, k\}$, the phylogenetic network $\cN_i$ is obtained from $\cN_{i-1}$ by cutting a reticulated cherry, and then suppressing $u'$ and $v$, where $u'$ is a parent of $v$ that is not $u$ in $\cN$. A simple induction argument shows that exactly the same sequence of cherry reductions from $\cN=\cN_0$ to $\cN_i$ can be applied to $\cN'$ to obtain the cherry-reduction sequence
$$\cN'=\cN'_0, \cN'_1, \cN'_2, \ldots, \cN'_i,$$
where, for all $j\in \{0, 1, \ldots, i-1\}$, the phylogenetic network is obtained from $\cN_j$ by \blue{deleting $(u, v)$, and identifying $u$ and $v$,} and $\cN'_i\cong \cN_i$. Using the cherry-reduction sequence from $\cN_{i+1}$ to $\cN_k$, it now follows that there exists a complete cherry-reduction sequence for $\cN'$, and so $\cN'$ is orchard.
\end{proof}

\blue{The next corollary is an immediate consequence of Lemma~\ref{lemo1}.}

\begin{corollary}
\blue{Let $\cN$ be an orchard network. Then $\id(\cN)$ is an orchard network.}
\label{stack1}
\end{corollary}

\blue{We next} describe two operations on sets of certain ordered pairs. These operations parallel the \blue{graph} operations of reducing cherries and cutting reticulated cherries. \blue{Intuitively, these operations explicitly describe how the ancestral profile of a phylogenetic networks changes if we reduce a cherry or cut a reticulated cherry (see Lemma~\ref{tuple1}).}

Let $X$ be a non-empty set and, for some fixed non-negative integer $t$, let
$$\Sigma=\{(x, \sigma(x)): x\in X\}$$
be a set of ordered pairs, where $\sigma(x)$ is a $t$-tuple each entry of which is either a non-negative integer or it is a \blue{placeholder} symbol $-$. \blue{Note that $\Sigma$ is an abstraction of $\Sigma_{\cN}$, where $\cN$ is a phylogenetic network.} We now describe two operations on $\Sigma$ that correspond to the two cherry-reduction operations. Let $\{a, b\}$ be a $2$-element subset of $X$. The first operation corresponds to reducing $b$ when $\{a, b\}$ is a cherry. Let $j\in \{1, 2, \ldots, t\}$ such that $\sigma_j(a)=\sigma_j(b)=1$, and $\sigma_j(x)=0$ for all $x\in X-\{a, b\}$. Let $\Sigma'$ be the set of $|X-\{b\}|$ ordered pairs obtained from $\Sigma$ as follows. For all $x\in X-\{b\}$, set $\sigma'(x)$ to be the $t$-tuple whose $i$-th entry is
$$\sigma'_i(x)=
\begin{cases}
\sigma_i(x), & \mbox{if $i\neq j$;} \\

- & \mbox{if $i=j$.}
\end{cases}
$$
Set $\Sigma'=\{(x, \sigma'(x): x\in X-\{b\}\}$. We say that $\Sigma'$ has been obtained from $\Sigma$ by {\em reducing} $b$.

The second operation corresponds to cutting $\{a, b\}$, when $\{a, b\}$ is a reticulated cherry with reticulation leaf $b$. Let $j\in \{1, 2, \ldots, t\}$ be such that $\sigma_j(a)=1=\sigma_j(b)$, and $\sigma_j(x)=0$ for all $x\in X-\{a, b\}$, and let $k\in \{1, 2, \ldots, t\}$ be such that $\sigma_k(b)=1$ and $\sigma_k(x)=0$ for all $x\in X-b$. The second operation has two types\footnote{In the correspondence of cutting a reticulated cherry $\{a, b\}$, the two types depend on whether or not the parent of $b$ is suppressed when cutting $\{a, b\}$.}. First, let $\Sigma'$ be the set of $|X|$ ordered pairs obtained from $\Sigma$ as follows. For all $x\in X-\{b\}$, set $\sigma'(x)$ to be the $t$-tuple whose $i$-th entry is
$$\sigma'_i(x)=
\begin{cases}
\sigma_i(x), & \mbox{if $i\not\in \{j, k\}$;} \\

-, & \mbox{if $i\in \{j, k\}$}
\end{cases}
$$
and set $\sigma'(b)$ to the $t$-tuple whose $i$-th entry is
$$\sigma'_i(b)=
\begin{cases}
\sigma_i(b)-\sigma_i(a), & \mbox{if $i\not\in \{j, k\}$;} \\

-, & \mbox{if $i\in \{j, k\}$.}
\end{cases}
$$
Set $\Sigma'=\{(x, \sigma'(x)): x\in X\}$. We say that $\Sigma'$ has been obtained from $\Sigma$ by {\em Type-I cutting} $\{a, b\}$.

Now let $\Sigma''$ be the set of $|X|$ ordered pairs obtained from $\Sigma$ as follows. For all $x\in X-\{b\}$, set $\sigma''(x)$ to be the $t$-tuple whose $i$-th entry is
$$\sigma''(x)=
\begin{cases}
\sigma_i(x), & \mbox{if $i\neq j$;} \\
-, & \mbox{if $i=j$}
\end{cases}$$
and set $\sigma''(b)$ to be the $t$-tuple whose $i$-th entry is
$$\sigma''(b)=
\begin{cases}
\sigma_i(b)-\sigma_i(a), & \mbox{if $i\neq j$;} \\
-, & \mbox{if $i=j$.}
\end{cases}$$
Set $\Sigma''=\{(x, \sigma''(x)): x\in X\}$. We say that $\Sigma''$ has been obtained from $\Sigma$ by {\em Type-II cutting} $\{a, b\}$.

The next lemma is established in~\cite[Lemma~5.1]{erd19} for binary phylogenetic networks. The extension to phylogenetic networks \blue{in which reticulations have in-degree at least two} is straightforward and omitted.

\begin{lemma}
Let $\cN$ be a phylogenetic network on $X$ with vertex set $V$ and $|X|\ge 2$, and fix an ordering of $V-X$. Let $\{a, b\}$ be a $2$-element subset of $X$.
\begin{enumerate}[{\rm (i)}]
\item If $\{a, b\}$ is a cherry of $\cN$, then, up to entries with symbol $-$, the set of ordered pairs obtained from $\Sigma_{\cN}$ by reducing $b$ is the ancestral profile of a phylogenetic network isomorphic to the phylogenetic network obtained from $\cN$ by reducing $b$.

\item Suppose that $\{a, b\}$ is a reticulated cherry of $\cN$ with reticulation leaf $b$. Then, up to entries with symbol $-$, the set of ordered pairs obtained from $\Sigma_{\cN}$ by
\begin{enumerate}[{\rm (I)}]
\item Type-I cutting $\{a, b\}$ is the ancestral profile of a phylogenetic network isomorphic to the phylogenetic network $\cN'$ obtained from $\cN$ by cutting $\{a, b\}$ in which the parent of $b$ is suppressed, and

\item Type-II cutting $\{a, b\}$ is the ancestral profile of a phylogenetic network isomorphic to the phylogenetic network $\cN$ obtained from $\cN$ by cutting $\{a, b\}$ in which the parent of $b$ is not suppressed.
\end{enumerate}
\end{enumerate}
\label{tuple1}
\end{lemma} 

Let $\cN$ be a phylogenetic network on $X$ with vertex set $V$, and let $v_1, v_2, \ldots, v_t$ be a fixed labelling of the vertices in $V-X$. For distinct $i, j\in \{1, 2, \ldots, t\}$, we say $v_i$ and $v_j$ are {\em clones} if, for all $x\in X$, we have $\sigma_i(x)=\sigma_j(x)$. \Blue{Characterising which pairs of vertices in an orchard network are clones is crucial to establishing Theorem~\ref{main2}. The next lemma gives this characterisation.}

\begin{lemma}
Let $\cN$ be an orchard network on $X$ with vertex set $V$. Let $v_1, v_2, \ldots, v_t$ be a fixed labelling of the vertices of $V-X$. Then $v_i$ and $v_j$ are clones if and only if one of the following holds:
\begin{enumerate}[{\rm (i)}]
\item $v_i$ and $v_j$ belong to the same sink of $\cN$; or

\item exactly one of $v_i$ and $v_j$ is a reticulation, say $v_i$, and there is a reticulation $v_k$ in the same sink of $\cN$ as $v_i$ such that $(v_k, v_j)$ is a (tree) arc of $\cN$.
\end{enumerate}
\label{clones1}
\end{lemma}

\blue{Before establishing Lemma~\ref{clones1}, we give an illustration of the lemma. Consider the orchard network shown in Fig.~\ref{orchard}. Every pair of vertices in $\{u, v, w\}$ are clones. Vertices $u$ and $v$ satisfy (i) of Lemma~\ref{clones1}, while vertices $u$ and $w$ (as well as $v$ and $w$) satisfy (ii) of Lemma~\ref{clones1}.}

\begin{proof}[Proof of Lemma~\ref{clones1}]
It is easily seen that if $v_i$ and $v_j$ are vertices for which either (i) or (ii) holds, then $v_i$ and $v_j$ are clones. For the converse, suppose that $v_i$ and $v_j$ are clones. The proof of the converse is by induction on the sum of the number $n$ of leaves and the number $r$ of reticulations of $\cN$. If $n+r=1$, then $n=1$ and $r=0$, and $\cN$ consists of a single vertex, and so the converse holds. If $n+r=2$, then, as $\cN$ is orchard, $n=2$ and $r=0$, and $\cN$ consists of two leaves adjoined to the root of $\cN$. Again, the converse holds.

Now assume that $n+r\ge 3$, so $n\ge 2$ \Blue{as $\cN$ is orchard}, and that the converse holds for all orchard networks in which the sum of the number of leaves and the number of reticulations is at most $n+r-1$. Since $\cN$ is orchard, $\cN$ has a $2$-element subset $\{a, b\}$ of $X$ such that $\{a, b\}$ is either a cherry or a reticulated cherry of $\cN$. First suppose that $\{a, b\}$ is a cherry of $\cN$. Let $p_a$ denote the common parent of $a$ and $b$. Let $\cN'$ denote the \blue{phylogenetic} network obtained from $\cN$ by reducing $b$. \blue{By Proposition~\ref{pro:orchard}, $\cN'$ is orchard.} Note that $\cN'$ has the same number of reticulations as $\cN$ but one less leaf. If $p_a\not\in \{v_i, v_j\}$, then, as $v_i$ and $v_j$ are clones of $\cN$, it follows by Lemma~\ref{tuple1}(i) that $v_i$ and $v_j$ are clones of $\cN'$. Thus, by induction, either (i) or (ii) holds in $\cN'$. In turn, this implies that either (i) or (ii) holds in $\cN$. Hence, without loss of generality, we may assume that $p_a=v_j$. Let $g_a$ be the (unique) parent of $p_a$ in $\cN$. If $g_a$ is a tree vertex, then there is a directed path from $g_a$ to a leaf $\ell$ such that $\ell\not\in \{a, b\}$. Since
$$\sigma_i(a)=\sigma_j(a)=1,$$
\Blue{that is, there is a path from $v_i$ to $a$ in $\cN$}, it follows that there is a path from $v_i$ to $g_a$, and so $\sigma_i(\ell)\ge 1$. But $\sigma_j(\ell)=0$, a contradiction as $v_i$ and $v_j$ are clones. Hence $g_a$ is a reticulation.

If $v_i$ belongs to the  sink $[g_a]$, then (ii) holds as $(g_a, v_j)$ is an arc of $\cN$. So assume that $v_i$ does not belong to the sink $[g_a]$. Since $\sigma_i(a)=\sigma_j(a)=1$, there is a path $P$ in $\cN$ from $v_i$ to $g_a$. Let $u$ denote the last tree vertex on $P$. Since $v_i\not\in [g_a]$, such a vertex exists and is the parent of a vertex in $[g_a]$. Then, as $u$ is a tree vertex, either there are at least two paths from $u$ to $a$ and so $\sigma_i(a)\ge 2$, a contradiction as $\sigma_j(a)=1$, or there is a path from $u$ to a leaf $\ell\not\in \{a, b\}$. But then $\sigma_i(\ell)\ge 1$ and $\sigma_j(\ell)=0$, another contradiction as $v_i$ and $v_j$ are clones.

Now suppose that $\{a, b\}$ is a reticulated cherry of $\cN$. Without loss of generality, we may assume that $b$ is the reticulation leaf. Let $p_a$ and $p_b$ denote the parents of $a$ and $b$, respectively. Let $\cN'$ be the \blue{phylogenetic} network obtained from $\cN$ by cutting $\{a, b\}$. \blue{By Proposition~\ref{pro:orchard}, $\cN'$ is orchard.} Note that $\cN'$ has the same number of leaves as $\cN$ but one less reticulation. If $\{p_a, p_b\}\cap \{v_i, v_j\}$ is empty, then, as $v_i$ and $v_j$ are clones of $\cN$, it follows by Lemma~\ref{tuple1}(ii) that $v_i$ and $v_j$ are clones of $\cN'$. Thus, by induction, either (i) or (ii) holds in $\cN'$, and therefore in $\cN$. If $\{p_a, p_b\}=\{v_i, v_j\}$, then either $\sigma_i(a)=1$ and $\sigma_j(a)=0$, or $\sigma_j(a)=1$ and $\sigma_i(a)=0$, a contradiction as $v_i$ and $v_j$ are clones. Thus we may assume that
$$|\{p_a, p_b\}\cap \{v_i, v_j\}|=1.$$

Without loss of generality, suppose that $v_j\in \{p_a, p_b\}$. Say $v_j=p_a$. Let $g_a$ be the (unique) parent of $p_a$ in $\cN$. Since $\sigma_i(a)=\sigma_j(a)=1$, it follows that there is a path from $v_i$ to $g_a$. If $g_a$ is a tree vertex, then there is a directed path from $g_a$ to a leaf $\ell$ not using $(g_a, p_a)$. Observe that $\ell\neq a$. If $\ell=b$, then $\sigma_i(b)\ge 2$, a contradiction as $\sigma_j(b)=1$. On the other hand, if $\ell\neq b$, then $\sigma_i(\ell)\ge 1$ and $\sigma_j(\ell)=0$, another contradiction. Thus we may assume $v_j=p_b$.

If $v_i$ belongs to the same sink $[p_b]$, then (i) holds. So assume $v_i$ does not belong to $[p_b]$. Since $\sigma_j(a)=0$, and $v_i$ and $v_j$ are clones, it follows that there is no path from $v_i$ to $p_a$. \Blue{However}, as $\sigma_j(b)=1$, there is a path $P$ from $v_i$ to $p_b$. Let $u$ denote the last tree vertex on $P$. Since $v_i\not\in [p_b]$ such a vertex exists and is the parent of a vertex in $[p_b]$. But then, as $u$ is a tree vertex, either there are at least two paths from $u$ to $b$, in which case $\sigma_i(b)\ge 2$, or there is a path from $u$ to a leaf $\ell\neq b$, in which case $\sigma_i(\ell)\ge 1$. Both cases contradict that $v_i$ and $v_j$ are clones as $\sigma_j(b)=1$ and $\sigma_j(x)=0$ for all $x\in X-\{b\}$. This completes the proof of the lemma.
\end{proof}
 
The next two results are consequences of Lemma~\ref{clones1}. The first result is immediate.

\begin{corollary}
Let $\cN$ be an \Blue{orchard} network on $X$ with vertex set $V$, and suppose that $\{v_1, v_2, v_3\}$ is a $3$-element subset of $V-X$. \green{If} $v_i$ and $v_j$ are clones for all distinct $i, j\in \{1, 2, 3\}$, \green{then} $\cN$ has a sink of size at least two, in which case at least two of the vertices in $\{v_1, v_2, v_3\}$ are in the same sink.
\label{clones2}
\end{corollary}

\Blue{Next let} $\cN$ be a phylogenetic network on $X$ with vertex set $V$, and let $v_i$ and $v_j$ be distinct vertices of $V-X$. We say $v_i$ and $v_j$ \Blue{are} a {\em maximal} pair of clones if $v_i$ and $v_j$ are clones, but there is no vertex $v_k\in V-(X\cup \{v_i, v_j\})$ such that every two elements in $\{v_i, v_j, v_k\}$ are clones.

\begin{corollary}
Let $\cN$ and $\cN'$ be orchard networks on $X$, and suppose $\Sigma_{\cN}=\Sigma_{\cN'}$. If $v_i$ and $v_j$ are a maximal pair of clones of $\cN$, then $v_i$ and $v_j$ are reticulations of $\cN$ if and only if $v_i$ and $v_j$ are reticulations of $\cN'$.
\label{maximal}
\end{corollary}

\begin{proof}
Since $\Sigma_{\cN}=\Sigma_{\cN'}$, the vertices $v_i$ and $v_j$ are a maximal pair of clones of $\cN'$. Thus, to prove the lemma, it suffices to show that if $v_i$ and $v_j$ are reticulations of $\cN$, then $v_i$ and $v_j$ are reticulations of $\cN'$. \Blue{Suppose that} $v_i$ and $v_j$ are reticulations of $\cN$. Then, as $v_i$ and $v_j$ \Blue{are} a maximal pair of clones, we may assume that $(v_i, v_j)$ is a stack arc of $\cN$ and $v_j$ is the parent of a leaf $\ell$. In particular, $\sigma_i(\ell)=\sigma_j(\ell)=1$ and $\sigma_i(x)=\sigma_j(x)=0$ for all $x\in X-\{\ell\}$. Now, if $v_i$ and $v_j$ are not reticulations of $\cN'$, then, without loss of generality, we may assume by Lemma~\ref{clones1} that $(v_i, v_j)$ is an arc of $\cN'$, in which $v_i$ is a reticulation and $v_j$ is a tree vertex. But then either there are at least two paths from $v_i$ to a leaf or there is a path from $v_i$ to a leaf that is not $\ell$. Both possibilities contradict the assumption that $\Sigma_{\cN}=\Sigma_{\cN'}$. This completes the proof of the corollary.
\end{proof}

Let $X$ be a non-empty finite set and, for some fixed integer $t$, let
$$\Sigma=\{(x, \sigma(x)): x\in X\}$$
be a set of ordered pairs, where $\Sigma(x)$ is a $t$-tuple whose entries are either non-negative integers or $-$ for all $x\in X$. We describe a further operation on $\Sigma$. This time the operation corresponds to \blue{deleting a stack arc and identifying its end vertices}. Let $j$, $k$, and $l$ be distinct element in $\{1, 2, \ldots, t\}$ such that
$$\sigma_j(x)=\sigma_k(x)=\sigma_l(x)$$
for all $x\in X$. Let $\Sigma'$ be the set of $|X|$ ordered pairs obtained from $\Sigma$ as follows. For all $x\in X$, set $\sigma'(x)$ to be the $t$-tuple whose $i$-th entry is
$$\sigma'_i(x)=
\begin{cases}
\sigma_i(x), & \mbox{if $i\neq j$;} \\
-, & \mbox{if $i=j$.}
\end{cases}$$
Set $\Sigma'=\{(x, \sigma'(x)): x\in X\}$. We say that $\Sigma'$ has been obtained from $\Sigma$ by \blue{{\em identifying}} $j$. The proof of the next lemma is routine and omitted.

\begin{lemma}
Let $\cN$ be a phylogenetic network on $X$ with vertex set $V$, and fix an ordering of $V-X$. Suppose that $(v_j, v_k)$ is a stack arc of $\cN$. Then, up to entries with symbol $-$, the set of ordered pairs obtained from $\Sigma_{\cN}$ by \blue{identifying} $j$ is the ancestral profile of a phylogenetic network isomorphic to the phylogenetic network obtained from $\cN$ by \blue{deleting} $(v_j, v_k)$, \blue{and identifying $v_j$ and $v_k$}.
\label{tuple2}
\end{lemma}

\begin{lemma}
\label{lemo2}
Let  $\cN$ and $\cN'$ be orchard networks on $X$ .
If  $\Sigma_{\cN} = \Sigma_{\cN'}$, then $\Sigma_{\id(\cN)} = \Sigma_{\id(\cN')} $.
\end{lemma}

\begin{proof}
Let $V$ denote the vertex set of $\cN$, and suppose that $\Sigma_{\cN}=\Sigma_{\cN'}$. Let $v_1, v_2, \ldots, v_t$ be a fixed labelling of the vertices of $\cN$ in $V-X$. Note that, as $\Sigma_{\cN}=\Sigma_{\cN'}$, the total number of vertices in $\cN$ and $\cN'$ is $t+|X|$. The proof is by induction on the number $s$ of stack arcs of $\cN$. If $s=0$, then $\cN=\id(\cN)$ and so, by Lemma~\ref{clones1}, if $v_i$ and $v_j$ are clones of $\cN$, then exactly one of $v_i$ and $v_j$ is a reticulation, say $v_i$, and $(v_i, v_j)$ is a tree arc of $\cN$. In particular, all sink classes of $\cN$ have size one. We next show that $\cN'$ has no stack arcs.

If $\cN'$ has a stack arc $e$, then there exists either a $3$-element subset of $V-X$ such that every pair of elements are clones or the two end vertices of $e$ form a maximal pair of clones. Since $\Sigma_{\cN}=\Sigma_{\cN'}$, it follows by Corollaries~\ref{clones2} and~\ref{maximal} that $\cN$ has a sink of size two, a contradiction. Thus $\cN'$ has no stack arcs, and so $\cN'=\id(\cN')$. Hence $\Sigma_{\id(\cN)}=\Sigma_{\id(\cN')}$.

Now assume that $s\ge 1$ and that the lemma holds for all pairs of orchard networks on the same leaf sets, where one of the networks has at most $s-1$ stack arcs. Since $s\ge 1$, there exists a stack arc $(v_i, v_j)$ of $\cN$, in which case $v_i$ and $v_j$ belong to the same sink and are clones of $\cN$. Since $\Sigma_{\cN}=\Sigma_{\cN'}$, it follows by Corollaries~\ref{clones2} and~\ref{maximal} that $\cN'$ has a pair $v'_i$ and $v'_j$ \Blue{of clones}, where $(v'_i, v'_j)$ is a stack arc of $\cN'$ and, for all $x\in X$,
$$\sigma_i(x)=\sigma_j(x)=\sigma_{i'}(x)=\sigma_{j'}(x).$$
Let $\cN_1$ denote the \blue{directed graph} obtained from $\cN$ by \blue{deleting $(v_i, v_j)$, and identifying $v_i$ and $v_j$. By Lemma~\ref{lemo1}, $\cN_1$ is orchard. Similarly, let $\cN'_1$ denote the directed graph obtained from $\cN'$ by deleting $(v'_i, v'_j)$, and identifying $v'_i$ and $v'_j$. By Lemma~\ref{lemo1} again, $\cN'_1$ is orchard.} Then, as $\Sigma_{\cN}=\Sigma_{\cN'}$, we deduce by Lemma~\ref{tuple2} that $\Sigma_{\cN_1}=\Sigma_{\cN'_1}$. Since the number of stack arcs of $\cN_1$ is $s-1$, it follows by the induction assumption that
$$\Sigma_{\id(\cN_1)}=\Sigma_{\id(\cN'_1)}.$$
But $\id(\cN)\cong \id(\cN_1)$ and $\id(\cN')\cong \id(\cN'_1)$, and so $\Sigma_{\id(\cN)}=\Sigma_{\id(\cN')}$, thereby completing the proof of the lemma.
\end{proof}

\begin{proof}[Proof of Theorem~\ref{main2}]
Suppose that $\cN$ and $\cN'$ are orchard networks on $X$ with $\Sigma_{\cN}=\Sigma_{\cN'}$.  By \blue{Corollary~\ref{stack1}}, $\id(\cN)$ and $\id(\cN')$ are orchard networks and so, by Lemma~\ref{lemo2}, $\Sigma_{\id(\cN)}=\Sigma_{\id(\cN')}$. Thus, by Theorem~\ref{main1}, $\id(\cN)$ is isomorphic to $\id(\cN')$.
\end{proof}

\section{Concluding Comments}
\label{comments}

We end by raising two questions \blue{concerning orchard networks} that may be interesting for future work (even in the case where such networks are assumed to be binary). \blue{The first} question is whether or not Theorem~\ref{main2} remains true if one removes the requirement that $\cN'$ is an orchard network. Note that Theorem~\ref{main1} requires only that $\cN$ is an orchard network. A second question is whether orchard networks can be characterised succinctly in terms of forbidden subgraphs. For example, a binary phylogenetic network is tree-child if and only if it has no stack reticulations and no (tree) vertex that is the parent of two distinct reticulations~\cite{sem16}. The class of binary `tree-based' networks have also been characterised \Blue{in a similar} way~\cite{zha16}. Such `forbidden subgraph' characterisations have turned out to be particularly helpful in the study of these phylogenetic networks and we expect the same to apply in the study of orchard networks.

\section*{Appendix: Adjustments Required for the Proof of Theorem~\ref{main1}}

The following lemma replaces~\cite[Lemma 3.3]{erd19}, in which the requirement that the grandparents of $b$ (i.e.\ parents of the parent of $b$), are tree vertices was omitted. Without this extra constraint, the lemma does not hold; an example to illustrate this the phylogenetic network $\cN$ in Fig.~\ref{fig1} by taking $a= x_3$ and $b=x_4$.

\begin{lemma}
\label{lem_new}
Let $\cN$ be a phylogenetic network on $X$, and let $\{a, b\}$ be a $2$-element subset of $X$. Then $\{a, b\}$ is a reticulated cherry of $\cN$ in which $b$ is the reticulation leaf and all grandparents of $b$ are tree vertices if and only if
\begin{enumerate}[{\rm (i)}]
\item $\gamma(a)\subsetneq \gamma(b)$,

\item there is no $x\in X-\{b\}$ such that $\gamma(a)\subsetneq \gamma(x)$, and

\item $\left|\gamma(b)-\bigcup_{x\in X-\{b\}} \gamma(x)\right|=1$.
\end{enumerate}
\end{lemma}

\noindent The proof of Lemma~\ref{lem_new} follows the same argument as the original statement of the lemma.

In addition to Lemma~\ref{lem_new}, the proof of Theorem~\ref{main1} requires two further lemmas. The first replaces~\cite[Corollary 4.2]{erd19} (which is correct as stated) and the second is to connect Lemma~\ref{lem_new} with stack-free orchard networks.
\begin{lemma}
Let $\cN$ be a stack-free orchard network, and let $\{a, b\}$ be a cherry or a reticulated cherry of $\cN$. If $\cN'$ is obtained from $\cN$ by reducing $b$ if $\{a, b\}$ is a cherry or cutting $\{a, b\}$ if $\{a, b\}$ is a reticulated cherry, then $\cN'$ is a stack-free orchard network.
\end{lemma}

\begin{lemma}
Let $\cN$ be a stack-free orchard network. If $\{a, b\}$ is a reticulated cherry of $\cN$ in which $b$ is the reticulation leaf, then the grandparents of $b$ are tree vertices.
\end{lemma}

\noindent Using these replacement lemmas, together with~\cite[Proposition 4.1]{erd19} and~\cite[Lemma 5.1]{erd19} replaced by Proposition~\ref{pro:orchard} and Lemma~\ref{tuple1}, respectively, the proof of Theorem~\ref{main1} follows the same argument, {\em mutatis mutandis}, as~\cite[Theorem 2.2]{erd19}.

\end{document}